%% file: MTNS_2020_Malzer.tex
\documentclass{ifacconf}

\usepackage{natbib}        

\usepackage{amsmath}
\usepackage{amssymb}
\usepackage{mathrsfs}
\usepackage{microtype}
\usepackage{subcaption}

\usepackage{pgfplots}
 \pgfplotsset{compat=newest}
  \usetikzlibrary{plotmarks}
  \usepackage{grffile}
\usepackage{graphics,xcolor}
\definecolor{P285U}{cmyk}{0.89,0.43,0.0,0.0}
\definecolor{P285U_font}{cmyk}{0.89,0.43,0.0,0.4}
\definecolor{lgray}{cmyk}{0,0,0,0.2}
\definecolor{myblue}{cmyk}{100,75,0,0}
\definecolor{jkuBlue}{RGB}{4,110,152}
\definecolor{jkuBlue}{RGB}{0,120,170}
\definecolor{jkuCyan}{RGB}{100,180,190}
\definecolor{jkuYellow}{RGB}{230,195,35}
\definecolor{jkuGrey}{RGB}{125,130,140}
\definecolor{jkuDarkGrey}{RGB}{51,51,51}
\definecolor{jkuLightGreen}{RGB}{195,215,75}
\definecolor{jkuGreen}{RGB}{115,180,85}
\definecolor{jkuPurple}{RGB}{145,75,130}
\definecolor{jkuRed}{RGB}{205,90,80}

\usepackage{mdframed}

\newenvironment{proof}{\begin{pf}}{\end{pf}}
\begin{document}
\begin{frontmatter}

\title{Energy-Based In-Domain Control and Observer Design for Infinite-Dimensional Port-Hamiltonian Systems\thanksref{footnoteinfo}} 

\thanks[footnoteinfo]{This work has been supported by the Austrian Science Fund (FWF) under
grant number P 29964-N32.}

\author[First]{Tobias Malzer}
\author[Second]{Jes\'{u}s Toledo}
\author[Second]{Yann Le Gorrec}
\author[First]{Markus Sch\"{o}berl}

\address[First]{Institute of Automatic
Control and Control Systems Technology, Johannes Kepler University Linz,
Altenbergerstrasse 66, 4040 Linz, Austria (e-mail: \{tobias.malzer\_1, markus.schoeberl\}@jku.at).}

\address[Second]{FEMTO-ST, Univ. Bourgogne Franche-Comt\'{e}, CNRS, Besan\c{c}on, France, 
24 rue Savary, F-25000 Besan\c{c}on, France (e-mail: \{legorrec, jesus.toledo\}@femto-st.fr).}


\begin{abstract}                
In this paper, we consider infinite-dimensional port-Hamiltonian systems with in-domain actuation by means of an approach based on Stokes-Dirac structures as well as in a framework that exploits an underlying jet-bundle structure. In both frameworks, a dynamic controller based on the energy-Casimir method is derived in order to stabilise certain equilibrias. Moreover, we propose distributed-parameter observers deduced by exploiting damping injection for the observer error. Finally, we compare the approaches by means of an in-domain actuated vibrating string and show the equivalence of the control schemes derived in both frameworks.
\end{abstract}


\end{frontmatter}

\input{Text_MTNS}


\bibliography{my_bib}             
                                                   







\end{document}

%% file: Text_MTNS.tex
\section{Introduction}

Energy-based approaches are well established for the control design
of dynamical systems since they allow to exploit the physical properties
of the underlying system, see e.g. \cite{Schaft2000,Ortega2001}.
In this context, especially the port-Hamiltonian (pH) system representation,
which originally was designed for systems described by ordinary differential
equations (ODEs), has proven to be an appropriate tool. From a control-engineering
point of view, a major advantage of the pH-framework is the fact that
so-called power ports can be introduced, which can be used to incorporate
the externally supplied energy and therefore enable the interconnection
with other systems. Recently, a lot of research effort has been invested
in the adaptation of the pH-framework and energy-based control strategies
to systems governed by partial differential equations (PDEs).

However, it is worth stressing that the pH-system representation is
not unique. With respect to control-engineering purposes, in particular
two different frameworks have turned out to be especially suitable.
The main difference between the so-called Stokes-Dirac scenario, see
\cite{Schaft2002,Gorrec2005}, and the jet-bundle approach, see \cite{Ennsbrunner2005,Schoeberl2011,Schoeberl2014a},
is the choice of the variables (energy variables vs. derivative variables),
and consequently, the generation of the power ports strongly depends
on the chosen approach; see e.g. \cite{Schoeberl2013b} for a comparison
of these frameworks by means of the well-known Mindlin plate. In this
contribution, we intend to compare not only the system representations
based on these frameworks, but also the controller as well as the
observer design and draw some interesting conclusions.

Similar to \cite{Macchelli2004,Macchelli2017} and \cite{Schoeberl2011,Rams2017a},
where dynamic control schemes based on the well-known energy-Casimir
method for so-called boundary-control systems are addressed within
the Stokes-Dirac scenario and the jet-bundle framework, respectively,
in this paper the aim is to derive dynamic control laws, where we
focus on infinite-dimensional systems with in-domain actuation. While
the energy-Casimir method has been extended to this system class in
\cite{Malzer2019} for the jet-bundle approach, in this paper we discuss
the controller design for in-domain actuated systems within both frameworks
and aim at comparing the results.

The energy-Casimir method is working appropriately if the initial
conditions of the plant are known precisely; however it yields unsatisfactory
results for uncertain initial conditions, see e.g. \cite{Rams2017b},
where this problem is briefly discussed for a boundary-control system.
To overcome this obstacle, in this paper the energy-based control
law shall rely on a distributed-parameter observer. In the literature,
a lot of different observer-design schemes for the infinite-dimensional
scenario are available, see e.g. \cite{Smyshlyaev2005} for an approach
that relies on the so-called backstepping methodology or \cite{Schaum2016}
for a dissipativity-based observer design. Motivated by the fact that
we focus on pH-systems, in this contribution we intend to exploit
energy considerations with respect to the observer design.

Thus, the main contributions are as follows: i) within both considered
frameworks, we derive a Casimir-based control law for in-domain actuated
systems, see Section 4. ii) moreover, in Section 5, a distributed-parameter
observer is deduced within both considered frameworks by exploiting
energy considerations for the observer error. iii) we demonstrate
the capability of the proposed approach by means of an in-domain actuated
vibrating string, where the equivalence of both control laws deduced
in the different frameworks is shown, see Section 6.

\section{Notation and Preliminaries}

In this paper, we investigate distributed-parameter systems with a
$1$-dimensional spatial domain that shall be equipped with the coordinate
$z\in\left[0,L\right]$. The standard inner product on $L^{2}([0,L];\mathbb{R}^{n})$
is denoted by $\left\langle \cdot,\cdot\right\rangle _{L^{2}}$, while
a Sobolev space of order $p$ is indicated by $H^{p}([0,L];\mathbb{R}^{n})$.
In the following, we discuss some differential-geometric preliminaries.
Formulars are kept short and readable by applying tensor notation
and especially Einstein's convention on sums, where we will not indicate
the range of the indicess when they are clear from the context. Furthermore,
we use the standard symbols $\mathrm{d}$, $\wedge$ and $\rfloor$
for the exterior derivative, the exterior (wedge) product and the
natural contraction between tensor fields, respectively. Moreover,
we avoid the use of pull-back bundles. The set of all smooth functions
on a manifold $\mathcal{M}$ is denoted by $C^{\infty}(\mathcal{M})$.

The total manifold $\mathcal{E}$ of a bundle $\pi:\mathcal{E}\rightarrow\mathcal{B}$
is equipped with the coordinates $(z,x^{\alpha})$, where $x^{\alpha}$,
with $\alpha=1,\ldots,n$, denotes the dependent variables, while
the $1$-dimensional base manifold possesses the independent coordiante
$(z)$, and therefore, a bundle -- often denoted by $\pi:(z,x^{\alpha})\rightarrow(z)$
-- allows to easily distinguish between dependent and independent
variables. Note that a mathematical expression restricted to the boundary
of $\mathcal{B}$ is indicated by $(\cdot)|_{0}^{L}$. To be able
to introduce so-called derivative variables, we consider jet manifolds,
where for instance the $2$nd jet manifold $\mathcal{J}^{2}(\mathcal{E})$
possesses the coordinates $(z,x^{\alpha},x_{z}^{\alpha},x_{zz}^{\alpha})$,
with $x_{zz}^{\alpha}$ denoting the $2$nd-order derivative variable
or jet coordinate, i.e. the $2$nd derivative of $x^{\alpha}$ with
respect to $z$.

Next, we introduce the tangent bundle $\tau_{\mathcal{E}}=\mathcal{T}(\mathcal{E})\rightarrow\mathcal{E}$,
which is equipped with the coordinates $(z,x^{\alpha},\dot{z},\dot{x}^{\alpha})$
together with the fibre bases $\partial_{z}=\partial/\partial z$
and $\partial_{\alpha}=\partial/\partial x^{\alpha}$. Furthermore,
the vertical tangent bundle $\nu_{\mathcal{E}}:\mathcal{V}(\mathcal{E})\rightarrow\mathcal{E}$
possessing the coordinates $(z,x^{\alpha},\dot{x}^{\alpha})$ is a
subbundle of $\tau_{\mathcal{E}}$, and hence, a vertical vector field
$v=\mathcal{E}\rightarrow\mathcal{V}(\mathcal{E})$ is a section given
in local coordinates as $v=v^{\alpha}\partial_{\alpha}$ with $v^{\alpha}\in C^{\infty}(\mathcal{E})$.
Consequently, by means of the total derivative $d_{z}=\partial_{z}+x_{z}^{\alpha}\partial_{\alpha}+x_{zz}^{\alpha}\partial_{\alpha}^{z}+\ldots$,
with the abbreviation $\partial_{\alpha}^{z}=\partial/\partial x_{z}^{\alpha}$,
the $1$st prolongation of a vertical vector field is given as $j^{1}(v)=v^{\alpha}\partial_{\alpha}+d_{z}(v^{\alpha})\partial_{\alpha}^{z}$.

The so-called co-tangent bundle $\tau_{\mathcal{E}}^{*}:\mathcal{T}^{*}(\mathcal{E})\rightarrow\mathcal{E}$,
which possesses the coordinates $(z,x^{\alpha},\dot{z},\dot{x}_{\alpha})$
and the bases $\mathrm{d}z$ and $\mathrm{d}x^{\alpha}$, is a further
important differential geometric object and allows to define a one-form
$w:\mathcal{E}\rightarrow\mathcal{T}^{*}(\mathcal{E})$ as a section
given in local coordinates as $w=\breve{w}\mathrm{d}z+w_{\alpha}\mathrm{d}x^{\alpha}$
with $\breve{w},w_{\alpha}\in C^{\infty}(\mathcal{E})$. In this contribution,
we are interested in densities $\mathfrak{F}=\mathcal{F}\mathrm{d}z$
where the coefficients may depend on $1$st-order jet variables, i.e.
$\mathcal{F}\in C^{\infty}(\mathcal{J}^{1}(\mathcal{E}))$. In particular,
we focus on the formal change of the corresponding integrated quantity
$\mathscr{F}=\int_{0}^{L}\mathcal{F}\mathrm{d}z$ along solutions
of a generalised vertical vector field, where for the calculation
we exploit the so-called Lie derivative reading as $\mathrm{L}_{v}(w)$
for a differential form $w$. Hence, for $1$st-order densities the
formal change can be decomposed according to
\begin{equation}
\dot{\mathscr{F}}=\int_{0}^{L}\mathrm{L}_{j^{1}\left(v\right)}(\mathcal{F}\Omega)=\int_{0}^{L}v\rfloor\delta\mathfrak{F}+\left.(v\rfloor\delta^{\partial}\mathfrak{F})\right|_{0}^{L}\label{eq:decomposition}
\end{equation}
by using integration by parts and Stoke's theorem. In (\ref{eq:decomposition}),
the map $\delta\mathfrak{F}=\delta_{\alpha}\mathcal{F}\mathrm{d}x^{\alpha}\wedge\mathrm{d}z$
is called variational derivative and is given as $\delta_{\alpha}\mathcal{F}=\partial_{\alpha}\mathcal{F}-d_{z}(\partial_{\alpha}^{z}\mathcal{F})$
in local coordinates, while the boundary operator locally reads as
$\delta_{\alpha}^{\partial}\mathcal{F}=\partial_{\alpha}^{z}\mathcal{F}$.

\section{Port-Hamiltonian Framework}

As already mentioned, the pH-system representation in the infinite-dimensional
scenario is not unique. In the following, two different approaches
that have proven to be adequate frameworks in particular with respect
to control-engineering purposes are presented, where both approaches
rely on energy considerations. Although the structures of the considered
system representations are quite different -- stemming from the fact
that different state variables are used --, it should be stressed
that the governing physic, i.e. the underlying system of PDEs, is
the same.

\subsection{Geometric Approach based on Jet-Bundle Structure}

First, a pH-system representation that is particularly suitable for
systems that allow for a variational characterisation is discussed.
The approach exploits an underlying jet-bundle structure and makes
heavy use of a certain power-balance relation, see \cite{Ennsbrunner2005,Schoeberl2008a,Schoeberl2014}
for instance. To this end, we consider the bundle $\pi:(z,x^{\alpha})\rightarrow(z)$.
Then, a pH-system with $1$st-order Hamiltonian $\mathfrak{H}=\mathcal{H}\mathrm{d}z$,
i.e. $\mathcal{H}\in C^{\infty}(\mathcal{J}^{1}(\mathcal{E}))$, including
in- and outputs on the domain can be given as\begin{subequations}\label{eq:pH_sys_jetbundle}
\begin{align}
\dot{x} & =(\mathcal{J}-\mathcal{R})(\delta\mathfrak{H})+u\rfloor\mathcal{G}\label{eq:pH_sys_dynamics}\\
y & =\mathcal{G}^{*}\rfloor\delta\mathfrak{H}\,
\end{align}
\end{subequations}together with appropriate boundary conditions.
It should be stressed that the linear operators $\mathcal{J},\mathcal{R}:\mathcal{T}^{*}(\mathcal{E})\wedge\mathcal{T}^{*}(\mathcal{B})\rightarrow\mathcal{V}(\mathcal{E})$,
describing the internal power flow and the dissipation effects of
the system, respectively, as well as the input operator $\mathcal{G}:\mathcal{U}\rightarrow\mathcal{V}\left(\mathcal{E}\right)$
in general can be differential operators. However, in this contribution
it is sufficient to use bounded linear mappings, where the coefficients
of the skew-symmetric interconnection tensor $\mathcal{J}$ satisfy
$\mathcal{J}^{\alpha\beta}=-\mathcal{J}^{\beta\alpha}\in C^{\infty}(\mathcal{J}^{2}(\mathcal{E}))$,
while $\mathcal{R}^{\alpha\beta}=\mathcal{R}^{\beta\alpha}\in C^{\infty}(\mathcal{J}^{2}(\mathcal{E}))$
and $\left[\mathcal{R}^{\alpha\beta}\right]\geq0$ is valid for the
coefficient matrix of the symmetric and positive semidefinite dissipation
map $\mathcal{R}$. The input map $\mathcal{G}$, where the components
$\mathcal{G}_{\xi}^{\alpha}$ may depend (amongst others) on the spatial
coordinate $z$, enables to incorporate external inputs located within
the spatial domain. While we focus on systems with lumped inputs
$u^{\xi}\in\mathcal{U}$,  the output components $y_{\xi}\in\mathcal{Y}$
can be interpreted as distributed output densities because $\mathcal{G}_{\xi}^{\alpha}$
are the components of the adjoint output map $\mathcal{G}^{*}:\mathcal{T}^{*}(\mathcal{E})\wedge\mathcal{T}^{*}(\mathcal{B})\rightarrow\mathcal{Y}$
as well. Since the input bundle $\rho:\mathcal{U}\rightarrow\mathcal{J}^{2}(\mathcal{E})$
is dual to the output bundle $\varrho:\mathcal{Y}\rightarrow\mathcal{J}^{2}(\mathcal{E})$,
see \cite[Section 4]{Ennsbrunner2005} or \cite[Section 3]{Schoeberl2008a},
one can deduce the important relation
\begin{equation}
(u\rfloor\mathcal{G})\rfloor\delta\mathfrak{H}=u\rfloor(\mathcal{G}^{*}\rfloor\delta\mathfrak{H})=u\rfloor y\,.\label{eq:distributed_collocation}
\end{equation}
Thus, by replacing $\mathfrak{F}$ by $\mathfrak{H}$ in (\ref{eq:decomposition})
and substituting (\ref{eq:pH_sys_dynamics}) for $v$, for the system
class under consideration the formal change of the Hamiltonian functional
$\mathscr{H}=\int_{0}^{L}\mathcal{H}\mathrm{d}z$ follows to
\begin{equation}
\dot{\mathscr{H}}=-\int_{0}^{L}\mathcal{R}(\delta\mathfrak{H})\rfloor\delta\mathfrak{H}+\int_{0}^{L}u\rfloor y+\left.(v\rfloor\delta^{\partial}\mathfrak{H})\right|_{0}^{L}\,,\label{eq:H_p}
\end{equation}
where the first part describes the energy that is dissipated and the
remaining parts correspond to collocation on the domain as well as
on the boundary. Moreover, if we introduce a local representation
for (\ref{eq:pH_sys_jetbundle}) as\begin{subequations}\label{eq:pH_sys_JB_local}
\begin{align}
\dot{x}^{\alpha} & =(\mathcal{J}^{\alpha\beta}-\mathcal{R}^{\alpha\beta})\delta_{\beta}\mathcal{H}+\mathcal{G}_{\xi}^{\alpha}u^{\xi}\label{eq:pH_sys_loc_dyn_JB}\\
y_{\xi} & =\mathcal{G}_{\xi}^{\alpha}\delta_{\alpha}\mathcal{H}\,,
\end{align}
\end{subequations}with $\alpha,\beta=1,\ldots,n$ and $\xi=1,\ldots,m$,
the power-balance relation (\ref{eq:H_p}) can be stated as
\[
\dot{\mathscr{H}}=-\int_{0}^{L}\delta_{\alpha}(\mathcal{H})\mathcal{R}^{\alpha\beta}\delta_{\beta}(\mathcal{H})\mathrm{d}z+\int_{0}^{L}u^{\xi}y_{\xi}\mathrm{d}z+\left.(\dot{x}^{\alpha}\delta_{\alpha}^{\partial}\mathcal{H})\right|_{0}^{L}\,.
\]
At this point it should be mentioned that in this contribution we
consider systems with trivial boundary conditions, and therefore,
the boundary ports $(\dot{x}^{\alpha}\delta_{\alpha}^{\partial}\mathcal{H})|_{0}^{L}$
vanish for the considered systems. However, the boundary terms could
easily be determined by applying the boundary operator $\delta_{\alpha}^{\partial}$,
which will also play an important role for the determination of so-called
Casimir conditions in Subsection 4.1.

\subsection{Approach based on Stokes-Dirac Structure}

In this subsection, we present the pH-approach that relies on Stokes-Dirac
structures and is closely related to functional analysis, see e.g.
\cite{Gorrec2005,Jacob2012} for a detailed analysis regarding the
well-posedness and stability of boundary-controlled pH-systems. Here,
we consider pH-systems according to\begin{subequations}\label{eq:pH_sys_SD}
\begin{multline}
\partial_{t}\chi(z,t)=P_{1}\partial_{z}(\mathcal{Q}(z)\chi(z,t))+\ldots\\
+(P_{0}-G_{0})(\mathcal{Q}(z)\chi(z,t))+\mathcal{B}(z)u(t)\label{eq:pH_sys_dynamics_SD}
\end{multline}
together with the collocated output densities
\begin{align}
y & =\mathcal{B}^{T}(z)\mathcal{Q}(z)\chi(z,t)\label{eq:pH_sys_output_densities_SD}
\end{align}
and appropriate boundary conditions given as
\begin{equation}
W_{\mathcal{B}}\left[\begin{array}{cc}
f_{\partial}^{T}(t) & e_{\partial}^{T}(t)\end{array}\right]^{T}=0\,.\label{eq:boundary_conditions_SD}
\end{equation}
\end{subequations}In (\ref{eq:pH_sys_dynamics_SD}), the input $u\in\mathbb{R}^{m}$
depends on the time $t$ solely, i.e. we restrict ourselves to lumped
inputs, whereas the input operator $\mathcal{B}(z)$ depends on the
spatial coordinate $z$, implying that the output densities (\ref{eq:pH_sys_output_densities_SD})
might be distributed over the spatial domain $z\in\left[0,L\right]$.
It is worth stressing that we intentionally use the same notation
for the input $u$ and the output $y$ as for (\ref{eq:pH_sys_jetbundle}),
since we assume that the inputs and outputs are the same in both frameworks.
However, as already mentioned, the two approaches mainly differ in
the choice of the variables which shall be highlighted by denoting
the system state $\chi\left(z,t\right)\in\mathbb{R}^{n}$. For the
matrices in (\ref{eq:pH_sys_dynamics_SD}) we have that $P_{1}=P_{1}^{T}\in\mathbb{R}^{n\times n}$
and is invertible, $P_{0}=-P_{0}^{T}\in\mathbb{R}^{n\times n}$, $\mathbb{R}^{n\times n}\ni G_{0}=G_{0}^{T}\geq0$,
and $\mathcal{Q}(\cdot)\in L^{2}([0,L];\mathbb{R}^{n\times n})$
denotes a bounded and continuously differentiable matrix-valued function,
where $\mathcal{Q}(z)=\mathcal{Q}^{T}(z)$ and $mI\leq\mathcal{Q}(z)\leq MI$,
with the constants $m,M>0$, is valid for all $z\in\left[0,L\right]$.
Note that $\chi$ and $\mathcal{Q}$ are often used instead of $\chi(z,t)$
and $\mathcal{Q}(z)$ for the sake of simplicity. Moreover, the state
space is $X=L^{2}([0,L];\mathbb{R}^{n})$, which is equipped with
the inner product $\left\langle \chi_{1},\chi_{2}\right\rangle _{\mathcal{Q}}=\left\langle \chi_{1},\mathcal{Q}\chi_{2}\right\rangle _{L^{2}}$
and the norm $\left\Vert \chi\right\Vert _{\mathcal{Q}}^{2}=\left\langle \chi,\chi\right\rangle _{\mathcal{Q}}$.
Moreover, the Hamiltonian can be given as $H(t)=\frac{1}{2}\left\Vert \chi\right\Vert _{\mathcal{Q}}^{2}$,
emphasising that the norm is related to the stored energy of a system.
Hence, $\chi(z,t)$ are called energy variables, while $\mathcal{Q}(z)\chi(z,t)$
represents the co-energy variables. To be able to reformulate the
boundary conditions of a system according to (\ref{eq:boundary_conditions_SD}),
the boundary port variables \cite[Eq. (2)]{Macchelli2017}
\begin{equation}
\left[\begin{array}{c}
f_{\partial}(t)\\
e_{\partial}(t)
\end{array}\right]=\underset{R}{\underbrace{\frac{1}{\sqrt{2}}\left[\begin{array}{cc}
P_{1} & -P_{1}\\
I & I
\end{array}\right]}}\left[\begin{array}{c}
\mathcal{Q}(L)\chi(L,t)\\
\mathcal{Q}(0)\chi(0,t)
\end{array}\right]\,,\label{eq:boundary_port_variables}
\end{equation}
which correspond to a linear combination of the co-energy variables
restricted to the boundary, are introduced, where $I$ denotes the
identity matrix of appropriate dimension. Hence, the time derivative
of $H(t)$ can be deduced to
\begin{equation}
\dot{H}=-\int_{0}^{L}(\mathcal{Q}\chi)^{T}G_{0}\mathcal{Q}\chi\mathrm{d}z+\int_{0}^{L}u^{T}y\mathrm{d}z+B\,,\label{eq:H_p_SD}
\end{equation}
with $B=\frac{1}{2}[(\mathcal{Q}\chi)^{T}P_{1}\mathcal{Q}\chi]_{0}^{L}$,
where the derivation given in \cite[Section 7.2]{Jacob2012} for boundary-control
systems without dissipation can be adopted in a straightforward manner.
Like (\ref{eq:H_p}), the balance equation (\ref{eq:H_p_SD}) decomposes
into dissipation (first term) and into collocation on the domain as
well as on the boundary.

\section{Energy-Based In-Domain Control}

The following section deals with the energy-based design of control
laws for infinite-dimensional systems with in-domain actuation within
both discussed approaches. In particular, we focus on systems with
lumped inputs, where the in-domain actuators exhibit a non-vanishing
spatial distribution, naturally requiring the use of finite-dimensional
controllers. Furthermore, it is of interest to compare the results
obtained within both approaches.

The main idea of the energy-Casimir method is to couple the plant
to a dynamic controller -- which will beneficially be given in a
pH-formulation -- in order to shape the total energy of the closed
loop and to inject damping into the system. The latter can be accomplished
either by means of controller states that are not related to the plant,
where the pH-structure of the controller shall be exploited, see e.g.
\cite{Rams2017a}, or by using an additional input $u'$ in the interconnection
of plant and controller given as
\begin{equation}
u=-y_{c}+u',\qquad u_{c}=\bar{y}\,.\label{eq:interconnection}
\end{equation}
In (\ref{eq:interconnection}), $u_{c}$ and $y_{c}$ denote the
in- and the output of the dynamic controller that will be declared
subsequently, while $\bar{y}$ corresponds to the integrated output
density of the plant.

\subsection{Jet-Bundle Approach}

The idea is now to construct a closed-loop system, by exploiting the
pH-system formulation, that exhibit a desired behaviour. To this
end, we consider a finite-dimensional dynamic controller given in
the pH-formulation\begin{subequations}\label{eq:pH_controller_JB}
\begin{align}
\dot{x}_{c}^{\alpha_{c}} & =(J_{c}^{\alpha_{c}\beta_{c}}-R_{c}^{\alpha_{c}\beta_{c}})\partial_{\beta_{c}}H_{c}+G_{c,\xi}^{\alpha_{c}}u_{c}^{\xi}\,,\\
y_{c,\xi} & =G_{c,\xi}^{\alpha_{c}}\partial_{\alpha_{c}}H_{c}\,,\label{eq:pH_controller_output_JB}
\end{align}
\end{subequations}with $\alpha_{c},\beta_{c}=1,\ldots,n_{c}$ and
$\xi=1,\ldots,m$. If we use the interconnection (\ref{eq:interconnection}),
locally given as
\begin{equation}
u^{\xi}=-\delta^{\xi\eta}y_{c,\eta}+u'^{\xi},\quad u_{c}^{\xi}=\delta^{\xi\eta}\int_{0}^{L}y_{\eta}\mathrm{d}z\,,\label{eq:interconnection_JB}
\end{equation}
with the Kronecker-Delta symbol meeting $\delta^{\xi\eta}=1$ for
$\xi=\eta$ and $\delta^{\xi\eta}=0$ for $\xi\neq\eta$, we obtain
a closed-loop system with
\begin{equation}
\mathscr{H}_{cl}(x,x_{c})=\int_{0}^{L}\mathcal{H}\mathrm{d}z+H_{c}\left(x_{c}\right)\,.\label{eq:H_cl}
\end{equation}
Consequently, by using the interconnection (\ref{eq:interconnection_JB})
and taking the fact that we consider systems with in-domain actuation
solely -- i.e. no power flow takes place through the boundary ports
-- into account, the formal change of $\mathscr{H}_{cl}$ follows
to
\begin{multline*}
\dot{\mathscr{H}}_{cl}=-\int_{0}^{L}\delta_{\alpha}(\mathcal{H})\mathcal{R}^{\alpha\beta}\delta_{\beta}(\mathcal{H})\mathrm{d}z+\ldots\\
-\partial_{\alpha_{c}}(H_{c})R^{\alpha_{c}\beta_{c}}\partial_{\beta_{c}}(H_{c})+u'^{\xi}\int_{0}^{L}y_{\xi}\mathrm{d}z\,.
\end{multline*}
\begin{rem}\label{rem:Stability}It should be stressed that a comprehensive
proof of stability for systems governed by PDEs requires functional
analysis, whereas the focus of this contribution is on structural/geometric
considerations. Thus, no detailed stability investigations are presented
here. Nevertheless, in Section 6 we exemplarily sketch the necessary
procedure by means of the observer error of an in-domain actuated
vibrating string. \end{rem}To be able to use the closed-loop Hamiltonian
as Lyapunov candidate, it must be ensured that $\mathscr{H}_{cl}$
exhibits a minimum at the desired equilibrium point, entailing that
the controller-Hamiltonian has to be designed properly. Therefore,
we exploit Casimir functionals of the form
\begin{equation}
\mathscr{C}^{\lambda}=x_{c}^{\lambda}+\int_{0}^{L}\mathcal{C}^{\lambda}\mathrm{d}z\,,\label{eq:Casimir_JB}
\end{equation}
where it should be stressed that they in general may depend on $1$st-order
jet variables, i.e. $\mathcal{C}^{\lambda}\in C^{\infty}(\mathcal{J}^{1}(\mathcal{E}))$.
Thus, (\ref{eq:Casimir_JB}) has to fulfil $\dot{\mathscr{C}}^{\lambda}=0$
independently of the system-energy function in order to serve as conserved
quantity. Consequently, due to $\mathscr{C}^{\lambda}=\kappa^{\lambda}$,
where the constants $\kappa^{\lambda}=\mathscr{C}^{\lambda}|_{t=t_{0}}$
depend on the initial states of the plant and the controller, we can
express the controller states that are related to the plant as $x_{c}^{\lambda}=\kappa^{\lambda}-\int_{0}^{L}\mathcal{C}^{\lambda}\mathrm{d}z$.
Hence, if each controller state is related to the plant, we are able
to write $\mathscr{H}_{cl}=\mathscr{H}\left(x\right)+H_{c}\left(x\right)$
indicating that we are able to shape the minimum of the closed loop.
However, it should be mentioned that $\kappa^{\lambda}$ cannot be
determined exactly when the initial conditions of the plant are not
known precisely, which would yield an offset in the resulting control
law, implying a deviation regarding the desired equilibrium. To overcome
this drawback, we design a distributed parameter observer in Section
5, as the fact that the estimated state will converge to the real
one implies that these offset vanishes.

\begin{prop}\label{prop:Casimir_Conditions_JetBundle}Consider the
closed loop that results due to the interconnection of the plant (\ref{eq:pH_sys_JB_local})
and the controller (\ref{eq:pH_controller_JB}) by means of (\ref{eq:interconnection_JB})
with $u'^{\xi}=0$. Then, if the functionals (\ref{eq:Casimir_JB})
fulfil the conditions\begin{subequations}\label{eq:Casimir_Conditions_JB}
\begin{align}
(J_{c}^{\lambda\beta_{c}}-R_{c}^{\lambda\beta_{c}}) & =0\\
\delta_{\alpha}\mathcal{C}^{\lambda}(\mathcal{J}^{\alpha\beta}-\mathcal{R}^{\alpha\beta})+G_{c,\xi}^{\lambda}K^{\xi\eta}\mathcal{G}_{\eta}^{\beta} & =0\label{eq:Casimir_Cond_domain_JB}\\
\delta_{\alpha}\mathcal{C}^{\lambda}\mathcal{G}_{\xi}^{\alpha}K^{\xi\eta}G_{c,\eta}^{\alpha_{c}} & =0\label{eq:Casimir_Input_Obstacle_JB}\\
(\dot{x}^{\alpha}\delta_{\alpha}^{\partial}\mathcal{C}^{\lambda})|_{0}^{L} & =0\label{eq:Casimir_Cond_Boundary_JB}
\end{align}
\end{subequations}for $\lambda=1,\ldots,\bar{n}\leq n_{c}$, they
serve as conserved quantities.\end{prop}

\begin{proof}For the proof we refer to \cite{Malzer2019}.\end{proof}

\subsection{Stokes-Dirac Approach}

Next, we want to deduce a control law based on the energy-Casimir
method within the Stokes-Dirac framework. Thus, we consider a dynamic
controller given in the pH-form\begin{subequations}\label{eq:pH_controller_SD}
\begin{align}
\dot{v}_{c} & =(A_{c}-S_{c})Q_{c}v_{c}+B_{c}u_{c}\label{eq:pH_controller_dynamics_SD}\\
\check{y}_{c} & =B_{c}^{T}Q_{c}v_{c}\label{eq:pH_controller_output_SD}
\end{align}
\end{subequations}with $v_{c}\in\mathbb{R}^{n_{c}}$, where it should
be noted that the controller input $u_{c}\in\mathbb{R}^{m}$ is the
same as for (\ref{eq:pH_controller_JB}), but the output $\check{y}_{c}\in\mathbb{R}^{m}$
might be different as (\ref{eq:pH_controller_output_JB}). In (\ref{eq:pH_controller_dynamics_SD}),
we have the Hamiltonian $\bar{H}_{c}=\frac{1}{2}v_{c}^{T}Q_{c}v_{c}$,
$A_{c}=-A_{c}^{T}$ and $S_{c}=S_{c}^{T}\geq0$. The interconnection
of (\ref{eq:pH_controller_SD}) and the plant (\ref{eq:pH_sys_SD})
according to (\ref{eq:interconnection}) enables that the closed-loop
system can be written as a pH-system characterised by the Hamiltonian
$H_{cl}(\chi,v_{c})=\frac{1}{2}\left\Vert \chi\right\Vert _{\mathcal{Q}}^{2}+\frac{1}{2}v_{c}^{T}Q_{c}v_{c}$.
To be able to porperly shape $H_{cl}(\chi,v_{c})$, we introduce Casimir
functions of the form
\begin{equation}
C=\Gamma^{T}v_{c}+\int_{0}^{L}\Psi^{T}\left(z\right)\chi\mathrm{d}z\,,\label{eq:Casimir_SD}
\end{equation}
representing a special case of (\ref{eq:Casimir_JB}) if $\Gamma$
is a unit vector.

\begin{prop}\label{prop:Casimir_Conditions_StokesDirac}The functionals
(\ref{eq:Casimir_SD}) serve as structural invariants of the closed
loop, stemming from the interconnection of the plant (\ref{eq:pH_sys_SD})
and the controller (\ref{eq:pH_controller_SD}) via (\ref{eq:interconnection})
with $u'=0$, if they meet the conditions\begin{subequations}\label{eq:Casimir_Conditions_SD}
\begin{align}
(A_{c}+S_{c})\Gamma & =0\label{eq:Casimir_cond_controller_SD}\\
P_{1}\partial_{z}\Psi+\left(P_{0}+G_{0}\right)\Psi-\mathcal{B}B_{c}^{T}\Gamma & =0\label{eq:Casimir_cond_domain_SD}\\
B_{c}\mathcal{B}^{T}\Psi & =0\label{eq:Casimir_Input_Obstacle_SD}\\
\left[\begin{array}{cc}
e_{\partial} & f_{\partial}\end{array}\right]R\left[\begin{array}{cc}
\Psi(L) & \Psi(0)\end{array}\right]^{T} & =0\label{eq:Casimir_Cond_Boundary_SD}
\end{align}
\end{subequations}\end{prop}

\begin{proof}Next, we show the proof of Prop. \ref{prop:Casimir_Conditions_StokesDirac},
which is a trivial adaptation of the boundary-control case, and hence,
we deduce the formal change of (\ref{eq:Casimir_SD}). Therefore,
if we substitute (\ref{eq:pH_sys_SD}), (\ref{eq:pH_controller_SD})
and the interconnection (\ref{eq:interconnection}) in $\dot{C}=\Gamma^{T}\dot{v}_{c}+\int_{a}^{b}\Psi^{T}\partial_{t}\chi\mathrm{d}z$,
an integration by parts yields 
\begin{multline*}
\dot{C}=\Gamma^{T}(A_{c}-S_{c})Q_{c}v_{c}+\ldots\\
\int_{0}^{L}(\Gamma^{T}B_{c}\mathcal{B}^{T}-(\partial_{z}\Psi)^{T}P_{1}+\Psi^{T}(P_{0}-G_{0}))\mathcal{Q}x\mathrm{d}z+\ldots\\
-\int_{0}^{L}\Psi^{T}\mathcal{B}B_{c}^{T}Q_{c}v_{c}\mathrm{d}z+(\Psi^{T}P_{1}\mathcal{Q}x)|_{0}^{L}\,.
\end{multline*}
Thus, the conditions (\ref{eq:Casimir_cond_controller_SD}), (\ref{eq:Casimir_cond_domain_SD})
and (\ref{eq:Casimir_Input_Obstacle_SD}) follows immediately by considering
the properties of $A_{c}$, $S_{c}$, $P_{1}$, $P_{0}$, $G_{0}$.
If we rewrite the expression restricted to the boundary as
\[
(\Psi^{T}P_{1}\mathcal{Q}x)|_{0}^{L}=\left[\begin{array}{c}
\Psi(L)\\
\Psi(0)
\end{array}\right]^{T}\left[\begin{array}{cc}
P_{1} & 0\\
0 & -P_{1}
\end{array}\right]\left[\begin{array}{c}
\mathcal{Q}(L)\chi(L,t)\\
\mathcal{Q}(0)\chi(0,t)
\end{array}\right]\,,
\]
and use the relation $R^{T}\Sigma R=\left[\begin{array}{cc}
P_{1} & 0\\
0 & -P_{1}
\end{array}\right]$, by means of (\ref{eq:boundary_port_variables}) we are able to deduce
condition (\ref{eq:Casimir_Cond_Boundary_SD}).\end{proof}

Although the considered approaches are quite different, in fact the
conditions (\ref{eq:Casimir_Conditions_SD}) yield a similar result
as Prop. \ref{prop:Casimir_Conditions_JetBundle}. In particular,
the conditions (\ref{eq:Casimir_cond_domain_SD}) and (\ref{eq:Casimir_Cond_domain_JB})
allows to relate the plant within the domain to the controller, which
is different compared to the Casimir conditions for boundary-control
systems given in \cite[Prop. 3.1]{Macchelli2017}. Like (\ref{eq:Casimir_Input_Obstacle_JB}),
condition (\ref{eq:Casimir_Input_Obstacle_SD}) implies that we cannot
find Casimir functions depending on system states where an input appears
in the corresponding system equation.

\section{Observer Design}

As already mentioned, an uncertain initial configuration of the plant
would cause some problems regarding the Casimir-based control law.
To be able to apply the control scheme anyway, we propose a distributed-parameter
observer in both considered frameworks. The idea is to design an observer
based on energy considerations such that the observer-error system
exhibits a desired behaviour. 

\subsection{Jet-Bundle Approach}

Now, the objective is to design an observer system by exploiting an
underlying jet-bundle structure such that the observer error tends
to zero. To this end, we introduce the bundle $\hat{\pi}:\hat{\mathcal{E}}\rightarrow\hat{\mathcal{B}}$
with coordinates $(\hat{z},\hat{x}^{\hat{\alpha}})$ for $\hat{\mathcal{E}}$
and $(\hat{z})$ for $\hat{\mathcal{B}}$. Next, we extend the copy
of the plant (\ref{eq:pH_sys_JB_local}) by an error-injection term
exploiting the additional input
\begin{equation}
u_{o}^{\eta}=\delta^{\eta\xi}(\bar{y}_{\xi}-\hat{\bar{y}}_{\xi})\,,\label{eq:observer_input_JB}
\end{equation}
and consequently, the observer system can be written as\begin{subequations}\label{eq:observer_JB}
\begin{align}
\dot{\hat{x}}^{\hat{\alpha}} & =(\mathcal{J}^{\hat{\alpha}\hat{\beta}}-\mathcal{R}^{\hat{\alpha}\hat{\beta}})\delta_{\hat{\beta}}\hat{\mathcal{H}}+\mathcal{G}_{\xi}^{\hat{\alpha}}u^{\xi}+\mathcal{K}_{\eta}^{\hat{\alpha}}u_{o}^{\eta}\,,\label{eq:observer_JB_dynamics}
\end{align}
with $\hat{\alpha},\hat{\beta}=1,\ldots,n$ and $\xi,\eta=1,\ldots,m$,
by means of the observer-energy density $\hat{\mathcal{H}}$. In (\ref{eq:observer_input_JB}),
$\bar{y}_{\xi}$, which corresponds to the integrated output density
of the plant according to $\bar{y}_{\xi}=\int_{0}^{L}y_{\xi}\mathrm{d}z$,
is assumed to be available as measurement quantity, while $\hat{\bar{y}}_{\xi}$
represents the copy of the integrated plant output according to $\hat{\bar{y}}_{\xi}=\int_{0}^{L}\hat{y}_{\xi}\mathrm{d}z$
with
\begin{equation}
\hat{y}_{\xi}=\mathcal{G}_{\xi}^{\hat{\alpha}}\delta_{\hat{\alpha}}\hat{\mathcal{H}}\,.\label{eq:observer_JB_output_densiteis}
\end{equation}
\end{subequations}The aim is to design the observer gain $\mathcal{K}_{\eta}^{\hat{\alpha}}$
such that the observer error $\tilde{x}=x-\hat{x}$ tends to $0$.
If we substitute (\ref{eq:pH_sys_loc_dyn_JB}) and (\ref{eq:observer_JB_dynamics})
in $\dot{\tilde{x}}=\dot{x}-\dot{\hat{x}}$, we have\begin{subequations}\label{eq:observer_error_system_JB}
\begin{align}
\dot{\tilde{x}}^{\tilde{\alpha}} & =(\mathcal{J}^{\tilde{\alpha}\tilde{\beta}}-\mathcal{R}^{\tilde{\alpha}\tilde{\beta}})\delta_{\tilde{\beta}}\tilde{\mathcal{H}}-\mathcal{K}_{\xi}^{\tilde{\alpha}}u_{o}^{\xi}\,.
\end{align}
for the dynamics of the observer error. Thus, if we consider
\begin{equation}
\tilde{y}_{\xi}=-\mathcal{K}_{\xi}^{\tilde{\alpha}}\delta_{\tilde{\alpha}}\tilde{\mathcal{H}}\label{eq:observer_error_output_JB}
\end{equation}
\end{subequations}as collocated output densities, the observer error
(\ref{eq:observer_error_system_JB}) can be interpreted as a pH-system
with the error-Hamiltonian $\tilde{\mathscr{H}}=\int_{0}^{L}\tilde{\mathcal{H}}\mathrm{d}z$,
where it should be mentioned that $\tilde{\mathscr{H}}\neq\mathscr{H}+\hat{\mathscr{H}}$.
Thus, the formal change of $\tilde{\mathscr{H}}$ follows to
\[
\dot{\tilde{\mathscr{H}}}=-\int_{0}^{L}\left(\delta_{\tilde{\alpha}}(\tilde{\mathcal{H}})\mathcal{R}^{\tilde{\alpha}\tilde{\beta}}\delta_{\tilde{\beta}}(\tilde{\mathcal{H}})+\delta_{\tilde{\alpha}}(\tilde{\mathcal{H}})\mathcal{K}_{\xi}^{\tilde{\alpha}}\delta^{\xi\eta}(\bar{y}_{\eta}-\hat{\bar{y}}_{\eta})\right)\mathrm{d}z
\]
by means of (\ref{eq:observer_input_JB}). Therefore, by choosing
the components $\mathcal{K}_{\xi}^{\hat{\alpha}}$ properly, it is
possible to achieve an error system with $\dot{\tilde{\mathscr{H}}}\leq0$,
i.e. the observer error is non-increasing. Of course, with regard
to the observer error it is necessary to show that it converges to
$0$, which is discussed in Section 6 for a concrete example.

\subsection{Stokes-Dirac Approach}

Next, we intend to exploit the framework proposed in Subsection 3.2
in order to deduce a distributed-parameter observer for systems with
in-domain actuation and collocated measurement. Again, the main idea
is to extend the copy of the plant by an error-injection term --
where like in Subsection 5.1 it is assumed that $\bar{y}$ is available
as measurement quantity -- and determine the observer gain such that
the observer error converges to $0$. In general, the infinite-dimensional
observer system can be introduced as\begin{subequations}
\begin{multline}
\partial_{t}\hat{\chi}(z,t)=P_{1}\partial_{z}(\mathcal{Q}(z)\hat{\chi}(z,t))+(P_{0}-G_{0})(\mathcal{Q}(z)\hat{\chi}(z,t))\\
+\mathcal{B}(z)u(t)+\mathcal{L}(z)(\bar{y}-\hat{\bar{y}})\,,\label{eq:observer_dynamics_SD}
\end{multline}
together with the boundary conditions
\begin{equation}
W_{\mathcal{B}}\left[\begin{array}{cc}
\hat{f}_{\partial}^{T}(t) & \hat{e}_{\partial}^{T}(t)\end{array}\right]^{T}=0\,,
\end{equation}
\end{subequations}where the corresponding boundary port variables
can be deduced according to (\ref{eq:boundary_port_variables}) with
the observer state $\hat{\chi}$. In (\ref{eq:observer_dynamics_SD}),
we have the same $P_{1}$, $P_{0}$, $G_{0}$, $\mathcal{Q}$ and
$\mathcal{B}$ as defined for (\ref{eq:pH_sys_dynamics_SD}), and
$\mathcal{L}(z)\in L^{2}([0,L];\mathbb{R}^{n\times m})$ denotes
the observer gain. If we define the observer error $\tilde{\chi}=\chi-\hat{\chi}$,
the dynamics of the observer error can be formulated as pH-system\begin{subequations}
\begin{multline}
\partial_{t}\tilde{\chi}(z,t)=P_{1}\partial_{z}(\mathcal{Q}(z)\tilde{\chi}(z,t))+\ldots\\
+(P_{0}-G_{0})(\mathcal{Q}(z)\tilde{\chi}(z,t))-\mathcal{L}(z)(\bar{y}-\hat{\bar{y}})\,,
\end{multline}
together with the collocated output density
\begin{equation}
\tilde{y}=-\mathcal{L}^{T}(z)\mathcal{Q}\tilde{\chi}
\end{equation}
and the boundary conditions
\begin{equation}
W_{\mathcal{B}}\left[\begin{array}{cc}
\tilde{f}_{\partial}^{T}(t) & \tilde{e}_{\partial}^{T}(t)\end{array}\right]^{T}=0\,.\label{eq:observer_error_boundary_condition_SD}
\end{equation}
\end{subequations}To properly design $\mathcal{L}(z)$, we consider
the formal change of the error energy $\tilde{H}=\frac{1}{2}\left\Vert \tilde{\chi}\right\Vert _{\mathcal{Q}}^{2}$,
which follows to
\[
\dot{\tilde{H}}=-\int_{0}^{L}(\mathcal{Q}\tilde{\chi})^{T}G_{0}\mathcal{Q}\tilde{\chi}\mathrm{d}z-\int_{0}^{L}(\bar{y}-\hat{\bar{y}})^{T}\mathcal{L}^{T}\mathcal{Q}\tilde{\chi}\mathrm{d}z
\]
by taking the boundary conditions (\ref{eq:observer_error_boundary_condition_SD})
into account. The objective is to determine $\mathcal{L}(z)$ such
that the observer error $\tilde{\chi}$ tends to $0$, which is demonstrated
for an in-domain actuated vibrating string in the following example.

\section{Example: Vibrating String}

To illustrate the proposed approaches, we consider a vibrating string
actuated within the spatial domain that can be modelled according
to\begin{subequations}
\begin{equation}
\rho\left(z\right)\frac{\partial^{2}w}{\partial t^{2}}=\frac{\partial}{\partial z}(T\left(z\right)\frac{\partial w}{\partial z})+f(z,t)\,,\label{eq:vib_String_eom}
\end{equation}
where $w$ describes the vertical deflection of the string, $\rho\left(z\right)=\textrm{const}$
the mass density and $T\left(z\right)=\textrm{const}$ Young's modulus.
We assume that the string is clamped at $z=0$ and free at $z=L$,
i.e. the boundary conditions
\begin{align}
w\left(0,t\right) & =0\,,\quad T\frac{\partial w}{\partial z}\left(L,t\right)=0\label{eq:BC_VS}
\end{align}
\end{subequations}are valid. The distributed force $f(z,t)=g\left(z\right)u\left(t\right)$
shall be generated by an actuator behaving like a piezoelectric patch,
which is located between $z=z_{p1}$ and $z=z_{p2}$, mathematically
described by $g\left(z\right)=h(z-z_{p1})-h(z-z_{p2})$, with $h\left(\cdot\right)$
denoting the Heaviside function. In fact, the force-distribution on
the domain $z_{p1}\leq z\leq z_{p2}$ is supposed to be constant 
and is scaled by $u\left(t\right)$, which can be interpreted as
an input voltage applied to the actuator. The objective is to stabilise
the desired equilibrium
\begin{equation}
w^{d}=\begin{cases}
az & \text{for}\quad0\leq z<z_{p1}\\
-b(z-z_{p2})^{2}+c & \text{for}\quad z_{p1}\leq z<z_{p2}\\
c & \text{for}\quad z_{p2}\leq z\leq L
\end{cases}\,,\label{eq:equilibrium_w}
\end{equation}
with $a,b>0$ and $c=b(z_{p2}-z_{p1})^{2}+az_{p1}$.

\subsection{Jet-Bundle Approach}

First, we consider the bundle $\pi:(z,w,p)\rightarrow(z)$, where
we have introduced the generalised momenta $p=\rho\dot{w}$. Consequently,
the governing equation of motion reads as
\begin{equation}
\dot{p}=Tw_{zz}+g(z)u\,.\label{eq:VS_JB}
\end{equation}
Hence, by means of the Hamiltonian density $\mathcal{H}=\frac{1}{2\rho}p^{2}+\frac{1}{2}T(w_{z})^{2}$,
we are able to reformulate (\ref{eq:VS_JB}) according to
\begin{align*}
\left[\begin{array}{c}
\dot{w}\\
\dot{p}
\end{array}\right] & =\left[\begin{array}{cc}
0 & 1\\
-1 & 0
\end{array}\right]\left[\begin{array}{c}
\delta_{w}\mathcal{H}\\
\delta_{p}\mathcal{H}
\end{array}\right]+\left[\begin{array}{c}
0\\
g\left(z\right)
\end{array}\right]u,\quad y=g\left(z\right)\frac{p}{\rho}\,.
\end{align*}
The integrated plant-output $\bar{y}=\int_{0}^{L}g\left(z\right)\frac{p}{\rho}\mathrm{d}z$
corresponds to the current through the actuator and is available as
measurement quantity. Furthermore, the formal change of the Hamiltonian
functional $\mathscr{H}$ reduces to $\dot{\mathscr{H}}=\int_{0}^{L}g\left(z\right)\frac{p}{\rho}u\mathrm{d}z$
because of the boundary conditions (\ref{eq:BC_VS}).

\subsubsection{Control Design}

In the following, we are interested in designing a dynamic controller
by exploiting the energy-Casimir method proposed in Subsection 4.1.
As we only have one output of the plant, one controller state shall
be related to the plant. Moreover, due to the fact that we intend
to inject damping by an additional input $u'$, we do not further
extend the dimension of the controller, i.e. $n_{c}=1$. To be able
to shape the closed-loop energy, we choose the ansatz $\mathcal{C}^{1}=-g\left(z\right)w$,
which fulfils the conditions (\ref{eq:Casimir_Conditions_JB}) for
$G_{c}^{1}=1$ and allows for a relation between the plant and the
controller state according to
\begin{equation}
x_{c}^{1}=\int_{0}^{L}g\left(z\right)w\mathrm{d}z\label{eq:xc1}
\end{equation}
by choosing the initial controller states properly. Furthermore, the
conditions (\ref{eq:Casimir_Conditions_JB}) restrict the controller
dynamics to $\dot{x}_{c}^{1}=u_{c}$. If we set the controller-Hamiltonian
to $H_{c}=\frac{1}{2}c_{1}(x_{c}^{1}-x_{c}^{1,d}-\frac{u_{s}}{c_{1}})^{2}$,
where we use the constant $c_{1}>0$ and
\begin{equation}
x_{c}^{1,d}=\int_{0}^{L}g\left(z\right)w^{d}\mathrm{d}z\,,\label{eq:xc1d}
\end{equation}
the equilibrium (\ref{eq:equilibrium_w}) becomes a part of the minimum
of $\mathscr{H}_{cl}$. Worth stressing is the fact that we have
included the term with $u_{s}$ in $H_{c}$ because the equilibrium
(\ref{eq:equilibrium_w}) requires non-zero power. Furthermore, the
output of the dynamic controller follows to $y_{c}=c_{1}(x_{c}^{1}-x_{c}^{1,d}-\frac{u_{s}}{c_{1}})$.
If we use the interconnection $u=-y_{c}+u'$ and $u_{c}=\int_{0}^{L}y\mathrm{d}z$
together with the dissipative output feedback $u'=-c_{2}\bar{y}$,
where $c_{2}>0$, the formal change of $\mathscr{H}_{cl}$ can be
deduced to $\dot{\mathscr{H}}_{cl}=-c_{2}\bar{y}^{2}\leq0$. To allow
for a comparison with the controller that is derived within the Stokes-Dirac
scenario later, we state the total-control law, which can be given
as
\begin{equation}
u=-c_{1}(\int_{0}^{L}g\left(z\right)w\mathrm{d}z-\int_{0}^{L}g\left(z\right)w^{d}\mathrm{d}z)-c_{2}\bar{y}+u_{s}\label{eq:control_law_VS_JB}
\end{equation}
by substituting (\ref{eq:xc1}) and (\ref{eq:xc1d}).

\subsubsection{Observer Design}

Next, an observer for the in-domain actuated vibrating string can
be introduced as
\begin{align*}
\left[\begin{array}{c}
\dot{\hat{w}}\\
\dot{\hat{p}}
\end{array}\right] & =\left[\begin{array}{cc}
0 & 1\\
-1 & 0
\end{array}\right]\left[\begin{array}{c}
\delta_{\hat{w}}\hat{\mathcal{H}}\\
\delta_{\hat{p}}\hat{\mathcal{H}}
\end{array}\right]+\left[\begin{array}{c}
0\\
g\left(z\right)
\end{array}\right]u+\left[\begin{array}{c}
k_{1}\\
k_{2}
\end{array}\right](\bar{y}-\hat{\bar{y}})
\end{align*}
by means of the observer density $\hat{\mathcal{H}}=\frac{1}{2\rho}\hat{p}^{2}+\frac{1}{2}T(\hat{w}_{z})^{2}$,
where $\hat{\bar{y}}=\int_{0}^{L}g(z)\frac{\hat{p}}{\rho}\mathrm{d}z$
corresponds to the copy of the plant output. Next, we introduce the
error coordinates $\tilde{w}=w-\hat{w}$, $\tilde{p}=p-\hat{p}$,
implying that the error dynamics follows to\begin{subequations}\label{eq:observer_error_dynamics_VS}
\begin{align}
\dot{\tilde{w}} & =\frac{1}{\rho}\tilde{p}-k_{1}(\bar{y}-\hat{\bar{y}}),\quad\dot{\tilde{p}}=T\tilde{w}_{zz}-k_{2}(\bar{y}-\hat{\bar{y}})\,.
\end{align}
\end{subequations}Consequently, by means of the energy density of
the observer error $\tilde{\mathcal{H}}=\frac{1}{2\rho}\tilde{p}^{2}+\frac{1}{2}T(\tilde{w}_{z})^{2}$,
the dynamics of the observer error can be given in the pH-formulation\begin{subequations}\label{eq:observer_error_VS_JB}
\begin{align}
\left[\begin{array}{c}
\dot{\tilde{w}}\\
\dot{\tilde{p}}
\end{array}\right] & =\left[\begin{array}{cc}
0 & 1\\
-1 & 0
\end{array}\right]\left[\begin{array}{c}
\delta_{\tilde{w}}\tilde{\mathcal{H}}\\
\delta_{\tilde{p}}\tilde{\mathcal{H}}
\end{array}\right]-\left[\begin{array}{c}
k_{1}\\
k_{2}
\end{array}\right](\bar{y}-\hat{\bar{y}})\,,\\
\tilde{y} & =-\left[\begin{array}{cc}
k_{1} & k_{2}\end{array}\right]\left[\begin{array}{c}
\delta_{\tilde{w}}\tilde{\mathcal{H}}\\
\delta_{\tilde{p}}\tilde{\mathcal{H}}
\end{array}\right]=k_{1}T\tilde{w}_{zz}-k_{2}\frac{\tilde{p}}{\rho}\,.
\end{align}
\end{subequations}A straightforward calculation of the formal change
of the error-Hamiltonian functional $\tilde{\mathscr{H}}$ yields
the relation
\begin{align*}
\dot{\tilde{\mathscr{H}}} & =\int_{0}^{L}(T\tilde{w}_{zz}k_{1}(\bar{y}-\hat{\bar{y}})-\frac{\tilde{p}}{\rho}k_{2}(\bar{y}-\hat{\bar{y}}))\mathrm{d}z\,.
\end{align*}
If we choose $k_{1}=0$ and $k_{2}=kg\left(z\right)$ with $k>0$,
by keeping in mind that $(\bar{y}-\hat{\bar{y}})=\int_{0}^{L}g(z)\frac{1}{\rho}\tilde{p}\mathrm{d}z$
holds, one can conclude that $\dot{\tilde{\mathscr{H}}}(\tilde{w},\tilde{p})=-k(\bar{y}-\hat{\bar{y}})^{2}\leq0$,
and therefore, the observer error is non-increasing. Hence, by initialising
the controller with $x_{c}^{1}(0)=\int_{0}^{L}g(z)\hat{w}(z,0)\mathrm{d}z$
and using $\dot{x}_{c}^{1}=u_{c}=\hat{\bar{y}}$ for the controller
dynamics, the control law (\ref{eq:control_law_VS_JB}) can be interpreted
as $u=-c_{1}(\int_{0}^{L}g\left(z\right)\hat{w}\mathrm{d}z-\int_{0}^{L}g\left(z\right)w^{d}\mathrm{d}z)-c_{2}\bar{y}+u_{s}$.
If $\hat{x}$ converges to $x$, the combination of controller and
observer stabilises the desired equilibrium exactly, and therefore,
in the following we sketch the procedure for proving well-posedness
and asymptotic stability of the observer error.

\subsubsection{Remarks on the Observer Convergence and Simulation Results}

 To analyse the well-posedness and stability of the observer-error
system, we introduce the state vector $\bar{w}=\left[w^{1},w^{2}\right]^{T}=\left[\tilde{w},\tilde{p}\right]^{T}$
together with the state space $\mathcal{W}=H_{C}^{1}(0,L)\times L^{2}(0,L)$,
where $H_{C}^{1}(0,L)=\{w^{1}\in H^{1}(0,L)|w^{1}(0)=0\}$. Furthermore,
we define the (energy) norm $\left\Vert \bar{w}\right\Vert _{\mathcal{W}}^{2}=T\left\langle \tilde{w}_{z},\tilde{w}_{z}\right\rangle _{L^{2}}+\frac{1}{\rho}\left\langle \tilde{p},\tilde{p}\right\rangle _{L^{2}}$,
where it can be shown that it is equivalent to the standard norm.
Next, we reformulate the observer-error dynamics as an abstract Cauchy
problem $\dot{\bar{w}}=\mathcal{A}\bar{w}$, with the operator $\mathcal{A}:\mathcal{D}(\mathcal{A})\subset\mathcal{W}\rightarrow\mathcal{W}$
given as
\[
\mathcal{A}:\left[\begin{array}{c}
\tilde{w}\\
\tilde{p}
\end{array}\right]\rightarrow\left[\begin{array}{c}
\rho^{-1}\tilde{p}\\
T\tilde{w}_{zz}-kg(z)\int_{0}^{L}g(z)\rho^{-1}\tilde{p}\mathrm{d}z
\end{array}\right]\,,
\]
where it can be shown that the inverse operator $\mathcal{A}^{-1}$
exists. Due to the relations $\tilde{\mathscr{H}}=\frac{1}{2}\left\Vert \bar{w}\right\Vert _{\mathcal{W}}^{2}$
and $\dot{\tilde{\mathscr{H}}}\leq0$ it follows that $\mathcal{A}$
is dissipative, and hence, we are able to apply a variant of the Lumer-Phillips
theorem \cite[Thm. 1.2.4]{Liu1999} in order to show that $\mathcal{A}$
generates a $C_{0}$-contraction semigroup. Furthermore, since $\mathcal{A}^{-1}$
is closed -- which can be shown similarly to the proof of Lemma 2.7
in \cite{Stuerzer2016} --, the solution trajectories are precompact
in $\mathcal{W}$, and consequently, we are able to apply LaSalle's
invariance principle (see \cite[Theorem 3.64, 3.65]{Luo1998}), leading
to a similar problem as in \cite[Section 3]{Guo2011}. In fact, it
can be shown that -- by choosing a proper length of the actuator
-- the only possible solution in $\mathcal{S}=\{\bar{w}\in\mathcal{W}|\dot{\tilde{\mathscr{H}}}=0\}$
is the trivial one, and hence, it follows that the observer error
is asymptotically stable. Moreover, in Fig. (\ref{fig:Comp_w_state_obs})
the comparison of the string deflection $w|_{L}$ and the corresponding
observer quantity $\hat{w}|_{L}$ is depicted for $k=30$, where the
observer is initialised with $\hat{w}(z,0)=cz$.
\begin{figure}
\input{w_L_Obs.tex}\caption{\label{fig:Comp_w_state_obs}Comparison of the string deflection $w|_{L}$
and the observer state $\hat{w}|_{L}$.}
\end{figure}
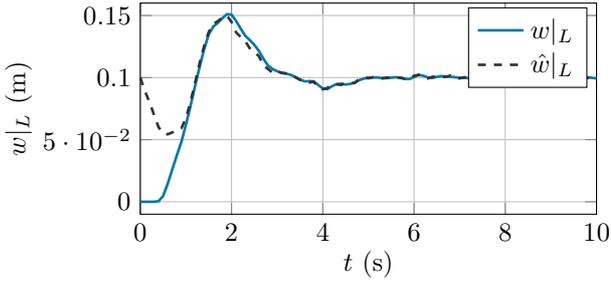
In Fig. (\ref{fig:deflection_VS}) it is shown that the propsed
controller in combination with the observer stabilises the desired
equilibrium (\ref{eq:equilibrium_w}) with $a=0.2$, $b=0.5$ and
$c=0.1$. Moreover, we set the string parameters $T$, $\rho$ and
$L$ to $1$. The actuator is placed between $z_{p1}=0.4$ and $z_{p2}=0.6$.
The controller parameters are chosen as $c_{1}=5$ and $c_{2}=30$.
Note that the finite difference-coefficient method has been applied
as discretisation scheme for the plant as well as for the implementation
of the observer.
\begin{figure}
\input{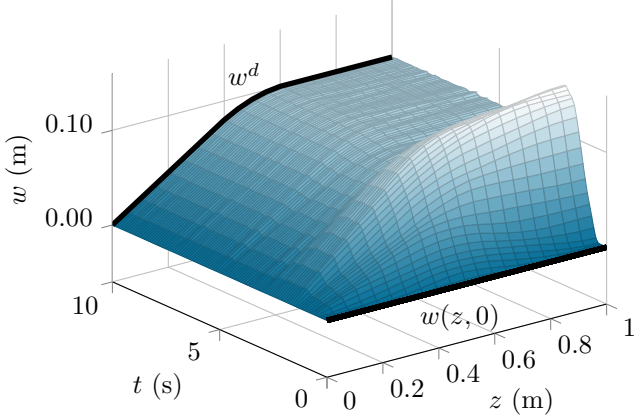}\caption{\label{fig:deflection_VS}Simulation result for the string deflection
$w$ over time $t$ and spatial domain $z$.}

\end{figure}

\subsection{Stokes-Dirac Approach}

As already mentioned, this framework relies on the use of energy variables,
and therefore, the strain $q=\partial_{z}w$ is introduced instead
of the deflection $w$, i.e. we use the state $\chi=\left[q,p\right]^{T}$
to describe the underlying system. Hence, an alternative pH-system
representation for the in-domain actuated vibrating string (\ref{eq:vib_String_eom})
reads as
\begin{equation}
\left[\begin{array}{c}
\dot{q}\\
\dot{p}
\end{array}\right]=\left[\begin{array}{cc}
0 & 1\\
1 & 0
\end{array}\right]\partial_{z}\left(\left[\begin{array}{cc}
T & 0\\
0 & \rho^{-1}
\end{array}\right]\left[\begin{array}{c}
q\\
p
\end{array}\right]\right)+\left[\begin{array}{c}
0\\
g(z)
\end{array}\right]u\,.\label{eq:sys_rep_SD}
\end{equation}
Consequently, by means of (\ref{eq:boundary_port_variables}) we are
able to determine the boundary-port variables
\[
\left[\begin{array}{c}
f_{\partial}\\
e_{\partial}
\end{array}\right]=\frac{1}{\sqrt{2}}\begin{tiny}\left[\begin{array}{c}
\rho^{-1}(L)p(L)-\rho^{-1}(0)p(0)\\
T(L)q(L)-T(0)q(0)\\
T(L)q(L)+T(0)q(0)\\
\rho^{-1}(L)p(L)+\rho^{-1}(0)p(0)
\end{array}\right]\end{tiny}\,,
\]
which further can be used to reformulate the boundary conditions (\ref{eq:BC_VS})
according to (\ref{eq:boundary_conditions_SD}) with
\begin{equation}
W_{\mathcal{B}}=\frac{1}{\sqrt{2}}\left[\begin{array}{cccc}
-1 & 0 & 0 & 1\\
0 & 1 & 1 & 0
\end{array}\right]\,.\label{eq:WB}
\end{equation}
 Thus, the formal change of the Hamiltonian $H=\frac{1}{2}\int_{0}^{L}(\frac{1}{\rho}p^{2}+Tq^{2})\mathrm{d}z$
follows to $\dot{H}=\int_{0}^{L}g(z)\frac{p}{\rho}u\mathrm{d}z$.

\subsubsection{Control Design}

Now, we focus on deriving a control law based on the method presented
in Subsection 4.2. Again, a dynamic controller with dimension one
is sufficient to reach our control objectives, and consequently, we
are interested in one Casimir function, i.e. we choose $\Gamma=1$.
Next, (\ref{eq:Casimir_Input_Obstacle_SD}) implies that $\Psi_{2}=0$
must be valid. Furthermore, if we set $B_{c}=1$, an evaluation of
the condition (\ref{eq:Casimir_cond_domain_SD}) yields the restriction
$\partial_{z}\Psi_{1}=g(z)$, where $\Psi_{1}$ has to be chosen such
that (\ref{eq:Casimir_Cond_Boundary_SD}) together with the boundary
conditions (\ref{eq:BC_VS}) is satisfied, i.e. $(\Psi_{1}\frac{p}{\rho})|_{0}^{L}=0$
must be valid. Thus, a relation between the plant and the controller
is given by
\begin{equation}
v_{c}=-\int_{0}^{L}\Psi_{1}q\mathrm{d}z\,,\label{eq:Casimir_VS_SD}
\end{equation}
and the dynamic of the controller is restricted to $\dot{v}_{c}=u_{c}$.
To achieve that the desired equilibrium (\ref{eq:equilibrium_w})
becomes a part of the minimum of $H_{cl}=H+\bar{H}_{c}$, we set
$\bar{H}_{c}=\frac{1}{2}d_{1}(v_{c}-v_{c}^{d}-\frac{u_{s}}{q_{c}})^{2}$,
with $d_{1}>0$ and (\ref{eq:Casimir_VS_SD}) with $q^{d}=\partial_{z}w^{d}$
instead of $q$ for $v_{c}^{d}$. If we use $u'=-d_{2}\bar{y}$ for
the damping injection, we obtain $\dot{H}_{cl}=-d_{2}\bar{y}^{2}$,
and the control law follows to
\begin{align}
u & =-d_{1}(-\int_{0}^{L}\Psi_{1}q\mathrm{d}z+\int_{0}^{L}\Psi_{1}q^{d}\mathrm{d}z)-d_{2}\bar{y}+u_{s}\,.\label{eq:control_law_VS_SD}
\end{align}

Next, it is of particular interest to compare the control laws (\ref{eq:control_law_VS_JB})
and (\ref{eq:control_law_VS_SD}). In fact, by substituting $q=\partial_{z}w$
and using integration by parts, we find that
\[
-\int_{0}^{L}\Psi_{1}\partial_{z}w\mathrm{d}z=-\left[\Psi_{1}w\right]_{0}^{L}+\int_{0}^{L}\partial_{z}(\Psi_{1})w\mathrm{d}z\,.
\]
Taking the boundary conditions and the restriction $\partial_{z}\Psi_{1}=g(z)$
into account, the condition (\ref{eq:Casimir_VS_SD}) can be rewritten
as $v_{c}=\int_{0}^{L}g\left(z\right)w\mathrm{d}z$. The same can
be done for $v_{c}^{d}=-\int_{0}^{L}\Psi_{1}q^{d}\mathrm{d}z$, and
consequently, we are able to give (\ref{eq:control_law_VS_SD}) as
\begin{equation}
u=-d_{1}(\int_{0}^{L}g\left(z\right)w\mathrm{d}z-\int_{0}^{L}g\left(z\right)w^{d}\mathrm{d}z)-d_{2}\bar{y}+u_{s}\,,\label{eq:control_law_SD}
\end{equation}
which is exactly the same control law as (\ref{eq:control_law_VS_JB})
if we set $c_{1}=d_{1}$ and $c_{2}=d_{2}$.

\subsubsection{Observer Design}

To be able to exploit the control scheme presented in the previous
subsection, in the following we design an observer according to the
strategy proposed in Subsection 5.2. Therefore, for the in-domain
actuated vibrating string, the observer can be introduced as
\[
\left[\begin{array}{c}
\dot{\hat{q}}\\
\dot{\hat{p}}
\end{array}\right]=\left[\begin{array}{cc}
0 & 1\\
1 & 0
\end{array}\right]\partial_{z}\left(\left[\begin{array}{cc}
T & 0\\
0 & \rho^{-1}
\end{array}\right]\left[\begin{array}{c}
\hat{q}\\
\hat{p}
\end{array}\right]\right)+\left[\begin{array}{c}
0\\
g
\end{array}\right]u+\left[\begin{array}{c}
l_{1}\\
l_{2}
\end{array}\right](\bar{y}-\hat{\bar{y}})\,,
\]
together with $W_{\mathcal{B}}\left[\begin{array}{c}
\hat{f}_{\partial}\\
\hat{e}_{\partial}
\end{array}\right]=\left[\begin{array}{c}
\rho(0)^{-1}\hat{p}(0)\\
T(L)\hat{q}(L)
\end{array}\right]=0$ by means of (\ref{eq:WB}). Thus, by using $\tilde{q}=q-\hat{q}$
and $\tilde{p}=p-\hat{p}$, we have
\[
\left[\begin{array}{c}
\dot{\tilde{q}}\\
\dot{\tilde{p}}
\end{array}\right]=\left[\begin{array}{cc}
0 & 1\\
1 & 0
\end{array}\right]\partial_{z}\left(\left[\begin{array}{cc}
T & 0\\
0 & \rho^{-1}
\end{array}\right]\left[\begin{array}{c}
\tilde{q}\\
\tilde{p}
\end{array}\right]\right)-\left[\begin{array}{c}
l_{1}\\
l_{2}
\end{array}\right](\bar{y}-\hat{\bar{y}})\,
\]
for the error dynamics and the collocated output reads as
\[
\tilde{y}=-\left[\begin{array}{cc}
l_{1} & l_{2}\end{array}\right]\left[\begin{array}{c}
T\tilde{q}\\
\rho^{-1}\tilde{p}
\end{array}\right]\,.
\]
Next, a careful investigation of the formal change of $\tilde{H}=\frac{1}{2}\int_{0}^{L}(\frac{1}{\rho}\tilde{p}^{2}+T\tilde{q}^{2})\mathrm{d}z$,
which can be deduced to
\begin{equation}
\dot{\tilde{H}}=-\int_{0}^{L}(l_{1}(\bar{y}-\hat{\bar{y}})T\tilde{q}+l_{2}(\bar{y}-\hat{\bar{y}})\frac{1}{\rho}\tilde{p})\mathrm{d}z\label{eq:H_p_error_VS_SD}
\end{equation}
by means of integration by parts and taking the boundary conditions
$W_{\mathcal{B}}\left[\begin{array}{c}
\tilde{f}_{\partial}\\
\tilde{e}_{\partial}
\end{array}\right]=\left[\begin{array}{c}
\rho(0)^{-1}\tilde{p}(0)\\
T(L)\tilde{q}(L)
\end{array}\right]=0$ into account, allows to determine the observer gains $l_{1}$ and
$l_{2}$ properly. In fact, the choice $l_{1}=0$ and $l_{2}=kg\left(z\right)$
-- i.e. we have exactly the same components for the observer gain
as in Subsection 6.1.2 -- renders (\ref{eq:H_p_error_VS_SD}) to
$\dot{\tilde{H}}=-k(\bar{y}-\hat{\bar{y}})^{2}\leq0$.

Let us finally mention again that the main difference of the proposed
approaches is the choice of the variables, cf. $x=\left[w,p\right]^{T}$
for the jet-bundle approach and $\chi=\left[q,p\right]^{T}$ for the
Stokes-Dirac approach, which implies that the Hamiltonian function
is different within the considered frameworks, although the energy
of course is the same. The choice of the variables also affect the
ansatz of the Casimir functions which might have a further impact
on the controller states; however, for the system under consideration
we have shown that we obtain the same control law, cf. (\ref{eq:control_law_VS_JB})
and (\ref{eq:control_law_SD}). Furthermore, although we considered
two completely different observer systems -- i.e. we have different
observer states -- we obtain the same dissipation rate for the observer
errors.

\section{Conclusion and Outlook}

In this paper, we have considered the controller and observer design
for infinite-dimensional pH-systems with in-domain actuation based
on a jet-bundle as well as on a Stokes-Dirac approach. In particular,
we have shown that -- although the underlying system descriptions
and therefore the controller and observer design are quite different
-- the proposed approaches yield the same results. As we only sketched
the proof of stability for the observer error, stability investigations
for the controller as well as for the combination with the observer
remain to be done.

%% file: w_L_Obs.tex
\begin{tikzpicture}

\begin{axis}[%
width=6cm,
height=2.8cm,
at={(0.0in,0.0in)},
xlabel style={yshift=0.1cm},
ylabel style={yshift=-0.1cm},
scale only axis,
xmin=0,
xmax=10,
xlabel={$t~(\text{s})$},
xmajorgrids,
ymin=-0.01,
ymax=0.16,
ylabel={$w|_{L}~(\text{m})$},
ymajorgrids,
axis background/.style={fill=white},
legend style={legend cell align=left,align=left,draw=black}
]

\addplot [color=jkuBlue,solid,line width=1.0pt]
  table[row sep=crcr]{%
0	0\\
0.1	1.32096346172423e-13\\
0.2	2.04642351843963e-08\\
0.3	1.14896619622216e-05\\
0.4	0.000512059498051376\\
0.5	0.00466789319270326\\
0.6	0.0144931866424321\\
0.7	0.0260122810061951\\
0.8	0.0372982221074747\\
0.9	0.0482042028291889\\
1	0.0620546887585488\\
1.1	0.0771068380129743\\
1.2	0.0931963247567162\\
1.3	0.108626324483215\\
1.4	0.123069458241493\\
1.5	0.134789996591847\\
1.6	0.141109105259123\\
1.7	0.144130143774684\\
1.8	0.147885690169963\\
1.9	0.151033404739036\\
2	0.150683687338495\\
2.1	0.145566842560852\\
2.2	0.139463699083744\\
2.3	0.133687278609978\\
2.4	0.130149958238749\\
2.5	0.126357462695275\\
2.6	0.121483212573534\\
2.7	0.115081402209043\\
2.8	0.109704011675966\\
2.9	0.106843350887945\\
3	0.105081075002252\\
3.1	0.103907635483112\\
3.2	0.10306471638099\\
3.3	0.10225043359956\\
3.4	0.100398102290234\\
3.5	0.0985608142988808\\
3.6	0.0969605201669917\\
3.7	0.0971557325922409\\
3.8	0.0970059856134233\\
3.9	0.0940276317283797\\
4	0.0912059016372829\\
4.1	0.0907183147898039\\
4.2	0.0930168450483697\\
4.3	0.0947059239567726\\
4.4	0.0950591215229145\\
4.5	0.094419550352293\\
4.6	0.0948990654965872\\
4.7	0.0959400105391529\\
4.8	0.0979824865111642\\
4.9	0.0992098623681997\\
5	0.0999010334054783\\
5.1	0.100007758213401\\
5.2	0.100059480901164\\
5.3	0.0995427586531467\\
5.4	0.0995239026501461\\
5.5	0.0999054025809462\\
5.6	0.100299588439014\\
5.7	0.10000466816442\\
5.8	0.0988971147015873\\
5.9	0.0988431357195226\\
6	0.100424701774005\\
6.1	0.10227647498451\\
6.2	0.101465418713527\\
6.3	0.100233529052875\\
6.4	0.0996752597695701\\
6.5	0.100597186857267\\
6.6	0.101180200332078\\
6.7	0.100906960934783\\
6.8	0.100688314442511\\
6.9	0.100278313118004\\
7	0.0998333472388037\\
7.1	0.0997522201146361\\
7.2	0.0996469638689949\\
7.3	0.0999215439691197\\
7.4	0.100045223985042\\
7.5	0.100018259280398\\
7.6	0.0996578440408692\\
7.7	0.0994262216990313\\
7.8	0.100128183413336\\
7.9	0.100588038155544\\
8	0.10093238273149\\
8.1	0.0992813420658049\\
8.2	0.0991147731461406\\
8.3	0.0994138082224999\\
8.4	0.100785058214696\\
8.5	0.100259685577856\\
8.6	0.0994697442627611\\
8.7	0.099772811799993\\
8.8	0.0995976662094128\\
8.9	0.100080723955278\\
9	0.0994460865538572\\
9.1	0.100249910578145\\
9.2	0.0999800355435301\\
9.3	0.100050900251757\\
9.4	0.0997877023278819\\
9.5	0.0998277257601999\\
9.6	0.0999745423000142\\
9.7	0.100106086807768\\
9.8	0.100343458341899\\
9.9	0.0997548901017225\\
10	0.0993335175965548\\
};

\addlegendentry{$w|_{L}$}
\addplot [color=jkuDarkGrey,dashed,line width=1.0pt]
  table[row sep=crcr]{%
0	0.1\\
0.1	0.0890923509591369\\
0.2	0.0804765879286487\\
0.3	0.070087151586299\\
0.4	0.0600903425768391\\
0.5	0.0550893846860285\\
0.6	0.0543439251108356\\
0.7	0.0558421357307384\\
0.8	0.0578002639165782\\
0.9	0.0592184502037966\\
1	0.0673429101995522\\
1.1	0.080297590246295\\
1.2	0.0958165479426486\\
1.3	0.111667623896279\\
1.4	0.125604548546399\\
1.5	0.136663336050889\\
1.6	0.143147674686778\\
1.7	0.145491767369608\\
1.8	0.148408458453672\\
1.9	0.149805871750661\\
2	0.144767533074895\\
2.1	0.136333246449129\\
2.2	0.132739946154138\\
2.3	0.129330548722276\\
2.4	0.123565032288519\\
2.5	0.120867271696078\\
2.6	0.117179567047812\\
2.7	0.10977390162181\\
2.8	0.10689417796313\\
2.9	0.103890085635254\\
3	0.105182006704655\\
3.1	0.103667618940175\\
3.2	0.103420035440249\\
3.3	0.10169801634955\\
3.4	0.100273186095196\\
3.5	0.0981369271894567\\
3.6	0.0971733037940809\\
3.7	0.0960180360787987\\
3.8	0.0965984767031129\\
3.9	0.0921901921313578\\
4	0.0905601339289446\\
4.1	0.0915938528646395\\
4.2	0.0947920506096003\\
4.3	0.0946203204718879\\
4.4	0.0940710494318604\\
4.5	0.0958235920809382\\
4.6	0.0955028247443654\\
4.7	0.0956888439941983\\
4.8	0.0989812585152426\\
4.9	0.100255115237326\\
5	0.099450899425663\\
5.1	0.100340275594824\\
5.2	0.0997074373471306\\
5.3	0.0993547261904192\\
5.4	0.100016823257072\\
5.5	0.0995703469605144\\
5.6	0.10028827320079\\
5.7	0.0997252060134931\\
5.8	0.0991589697459549\\
5.9	0.0987358771689181\\
6	0.101532277076927\\
6.1	0.102082292238115\\
6.2	0.101240343269493\\
6.3	0.0992883756575382\\
6.4	0.100589830847003\\
6.5	0.100998469759674\\
6.6	0.100510787008915\\
6.7	0.100638010698079\\
6.8	0.101637346921585\\
6.9	0.0995271835796149\\
7	0.0993638228449733\\
7.1	0.100305189556556\\
7.2	0.0997801715736208\\
7.3	0.0992229134019665\\
7.4	0.101078915538025\\
7.5	0.0990272284953368\\
7.6	0.100002613943671\\
7.7	0.099869472979802\\
7.8	0.0995662485355835\\
7.9	0.101211575634084\\
8	0.100467954138067\\
8.1	0.0991238978402305\\
8.2	0.0989336092299525\\
8.3	0.10015483015031\\
8.4	0.100566435033258\\
8.5	0.100014914482967\\
8.6	0.0993015146504244\\
8.7	0.100005792717878\\
8.8	0.100225183859547\\
8.9	0.0985652185217305\\
9	0.100689649324091\\
9.1	0.100310553234478\\
9.2	0.099237083496797\\
9.3	0.100331969553833\\
9.4	0.100267278770196\\
9.5	0.0989626116992087\\
9.6	0.100704807097463\\
9.7	0.100053041795702\\
9.8	0.0996916661524721\\
9.9	0.100489921295775\\
10	0.0988857063634776\\
};
\addlegendentry{$\hat{w}|_{L}$}
\end{axis}
\end{tikzpicture}%